\documentclass[a4paper,reqno]{amsart}

\usepackage{amssymb, longtable}
\usepackage[mathcal]{euscript}

\usepackage{color}

\usepackage{hyperref}                  

\newcommand{\bydef}{:=}

\newcommand{\cA}{\mathcal{A}}
\newcommand{\cB}{\mathcal{B}}
\newcommand{\cC}{\mathcal{C}}

\newcommand{\cI}{\mathcal{I}}
\newcommand{\cJ}{\mathcal{J}}
\newcommand{\cK}{\mathcal{K}}

\newcommand{\cQ}{\mathcal{Q}}

\newcommand{\cS}{\mathcal{S}}

\newcommand{\cU}{\mathcal{U}}
\newcommand{\cV}{\mathcal{V}}

\newcommand{\frg}{{\mathfrak g}}

\newcommand{\ZZ}{\mathbb{Z}}

\newcommand{\RR}{\mathbb{R}}
\newcommand{\CC}{\mathbb{C}}

\newcommand{\FF}{\mathbb{F}}

\newcommand{\chr}[1]{\mathrm{char}\,#1}

\DeclareMathOperator{\Aut}{\mathrm{Aut}}

\DeclareMathOperator{\Der}{\mathrm{Der}}
\DeclareMathOperator{\LocDer}{\mathrm{LocDer}}

\DeclareMathOperator{\LocAut}{\mathrm{LocAut}}
\DeclareMathOperator{\twoLocAut}{\mathrm{2LocAut}}

\newcommand{\frso}{{\mathfrak{so}}}

\newcommand{\Ort}{\mathrm{O}}

\newcommand{\SP}{\mathrm{Sp}}

\numberwithin{equation}{section}

\newcommand{\nup}{\textup{n}}

\newtheorem{theorem}{Theorem}[section]

\newtheorem*{proposition*}{Proposition}
\newtheorem{lemma}[theorem]{Lemma}

\theoremstyle{definition}

\newtheorem{problems}[theorem]{Problems}
\theoremstyle{remark} 

\begin{document}

\title[Local and $2$-local automorphisms of Cayley algebras]%
{Local and $2$-local automorphisms\\ of Cayley algebras}



\author[Sh. A. Ayupov]{Shavkat Ayupov$^{1,2}$}
\address{$^{1}$V.I.Romanovskiy Institute of Mathematics\\
	Uzbekistan Academy of Sciences\\ University street, 9, Olmazor district, Tashkent, 100174, Uzbekistan}
\address{$^2$National University of Uzbekistan \\
	4, Olmazor district, Tashkent, 100174, Uzbekistan}
\email{\textcolor[rgb]{0.00,0.00,0.84}{shavkat.ayupov@mathinst.uz}}

\author[A.Elduque]{Alberto Elduque$^3$}
\address{$^{3}$Departamento de
Matem\'{a}ticas e Instituto Universitario de Matem\'aticas y
Aplicaciones, Universidad de Zaragoza, 50009 Zaragoza, Spain}
\email{\textcolor[rgb]{0.00,0.00,0.84}{elduque@unizar.es}}
\thanks{${}^3$%
Supported by grants MTM2017-83506-C2-1-P (AEI/FEDER, UE) and E22\_17R
(Gobierno de Arag\'on, Grupo de referencia ``\'Algebra y
Geometr{\'\i}a'', cofunded by Feder 2014-2020 ``Construyendo Europa desde Arag\'on'')}

\author[K. Kudaybergenov]{Karimbergen Kudaybergenov$^{1,4}$}
\address{$^4$Department of Mathematics\\
	Karakalpak State University\\
	1, Ch. Abdirov,   Nukus, 230112, Uzbekistan}
\email{\textcolor[rgb]{0.00,0.00,0.84}{karim2006@mail.ru}}




\date{\today}

\begin{abstract}
	The present paper is devoted to the description  of local and $2$-local automorphisms on Cayley algebras over an arbitrary field $\FF$.
	Given a Cayley  algebra $\cC$
	with  norm $\nup$, let  $\Ort(\cC,\nup)$ be the corresponding orthogonal group.
	We prove that the group of all local automorphisms of $\cC$ coincides with the group $\{\varphi\in\Ort(\cC,\nup)\mid \varphi(1)=1\}.$
	Further we prove that the behavior of $2$-local
		automorphisms  depends on the Cayley algebra being split or division. Every
	$2$-local automorphism on the split Cayley algebra is an automorphism, i.e. they form the exceptional Lie group  $G_2(\FF)$ if
	$\textrm{char}\FF\neq 2,3$.
	On the other hand, on  division Cayley algebras over a  field $\FF$,  the groups of $2$-local automorphisms and local automorphisms coincide, and they are isomorphic to  the group $\{\varphi\in\Ort(\cC,\nup)\mid \varphi(1)=1\}.$
\end{abstract}


\maketitle

\section{Introduction}

Let $\cA$ be an   algebra (not necessary associative). Recall that a linear bijection
$\varphi: \cA\to \cA$ is said to be an automorphism, if $\varphi(xy)=\varphi(x)\varphi(y)$ for all $x, y\in \cA.$
A  linear mapping  $\psi$ is said to be a
local  automorphism, if for every $x\in \cA$ there exists an automorphism $\varphi_x$ on $\cA$ (depending on $x$) such that $\psi(x)=\varphi_x(x).$

The latter notion was  introduced  by D.~R.~Larson and A.~R.~Sourour~\cite{Larson90}.
They  proved that if $\cA=B(X),$ the algebra of all bounded linear operators on a Banach space $X,$ then every invertible  local automorphism of $\cA$ is  an automorphism. Thus automorphisms on $B(X)$ are completely determined by their local actions. In \cite[Lemma 4]{E11} it was shown that the set of all local automorphisms $\LocAut(\cA)$ of an algebra $\cA$ form a multiplicative group.

In \cite{BS93} M.~Bresar and P.~\v{S}emrl proved that linear mappings of matrix algebras which preserve idempotents are Jordan homomorphisms. Applying this result the authors  obtained some results concerning local derivations and local automorphisms. As an another application,  they give a complete description of all weakly continuous linear surjective mappings on standard operator algebras which preserve projections. In 2007, M. Bresar \cite{B07} proposed more general methods to study local automorphisms on associative rings,   and in certain rings containing non-central idempotents  he characterized homomorphisms,
derivations and multipliers by their actions on elements satisfying some special conditions. He considered the condition that an additive map $h$ between rings $\cA$
and $\cB$ satisfies $h(x)h(y)h(z)=0$ whenever $x,y,z\in \cA$
are such that $xy=yz=0.$
As an application, he proved some new and interesting results on local derivations and local multipliers.

In~\cite{AK18} it was proved that every local
automorphism on the special linear  Lie algebra $\mathfrak{sl}_n$ is an
automorphism or anti-automorphism. Further M.~Constantini \cite{C19}   extended  the above result to an arbitrary simple  Lie algebra $\mathfrak{g}.$
These results can be reformulate as follows
$
\LocAut(\mathfrak{g})=\Aut(\mathfrak{g})\ltimes \{\textrm{id}_\mathfrak{g}, -\textrm{id}_\mathfrak{g}\},
$
where $\textrm{id}_\mathfrak{g}$ is the identical automorphism of $\mathfrak{g}.$

In 1997, P.~\v{S}emrl \cite{Semrl97} introduced the concept of 2-local automorphism.
Recall that a mapping $\Delta:\cA \to \cA$ (not necessary linear) is said to be a $2$-local automorphism, if for every pair $x,y\in \cA$ there exists an  automorphism $\varphi_{x,y}:\cA \to \cA$ (depending on $x,y$) such that
$\Delta(x)=\varphi_{x,y}(x),$ $\Delta(y)=\varphi_{x,y}(y).$
In \cite{Semrl97} P. \v{S}emrl described 2-local automorphisms  on the algebra $B(H)$  of all bounded linear operators on the infinite-dimensional separable Hilbert space $H,$  by proving that every 2-local  automorphism on $B(H)$  is an  automorphism.

In \cite{AEK} we have proved that the space of all local derivations of
a Cayley algebra $\cC$ coincides with the Lie algebra
$\{d\in\frso(\cC,\nup)\mid d(1)=0\}$ which is isomorphic to the  orthogonal Lie algebra
$\frso(\cC_0,\nup)$ (even if the characteristic of $\FF$ is $2$!).
Further it was shown that the structure of  $2$-local
derivations depends on the Cayley algebra being split or division. Every
$2$-local derivation on the split Cayley algebra is a derivation, i.e. the set of 2-local derivations coincides with the exceptional Lie algebra $\frg_2(\FF)$ if
$\textrm{char}\FF\neq 2,3$.
On the other hand, on  division Cayley algebras over a  field $\FF$,  the sets of $2$-local derivations and local derivations coincide, and they are isomorphic to  the Lie algebra $\frso(\cC_0,\nup)$.
In \cite{AKA} general forms of local automorphisms were obtained
on a Cayley algebra $\cC$ over a field $\FF$ of characteristic zero and  for $2$-local automorphisms on $\cC$
over an algebraically closed field $\FF$ of characteristic zero.

In the present paper we shall give   a complete description  of local and $2$-local automorphisms on Cayley algebras over an arbitrary field $\FF$.

\medskip

\section{Cayley algebras}

Let $\FF$  be an arbitrary field. Cayley  (or octonion) algebras  over $\FF$
constitute a well-known class
of nonassociative algebras. They are unital nonassociative
algebras $\cC$  of dimension eight over $\FF$,
endowed with a nonsingular quadratic multiplicative form (the norm)
$\nup : \cC \to  \FF$. Hence
$$
\nup(xy)=\nup(x)\nup(y)
$$
for all  $x, y \in  \mathcal{C}$, and the polar form
$$
\nup(x, y)\bydef\nup(x+y)-\nup(x)-\nup(y)
$$
is a nondegenerate bilinear form. (The norm and its associated polarization will be denoted by the same letter).

Any element in a Cayley algebra $\cC$ satisfies the degree $2$ equation:
\begin{align}\label{xsqn}
	x^2-\nup(x,1)x+\nup(x)1=0.
\end{align}
The map $x \to  \overline{x}=\nup(x, 1)1-x$ is an involution
and the
trace $t(x)=\nup(x, 1)$ and norm $\nup(x)$ are given by respectively, $t(x)1=x+\overline{x}$ and
$\nup(x)1 = x\overline{x}=\overline{x}x$ for all $x \in \cC$.

Note that two Cayley algebras $\mathcal{C}_1$ and $\mathcal{C}_2$, with respective norms $\nup_1$
and $\nup_2$, are isomorphic if and only if the norms  $\nup_1$ and $\nup_2$ are
isometric (see \cite[Corollary 4.7]{EKmon}).  It is necessary to remind that Cayley
algebras are alternative, that is, each subalgebra generated by any two elements
	is associative.

From \eqref{xsqn}, we get
\begin{align*}
	xy+yx-\nup(x,1)y-\nup(y,1)x +\nup(x,y)1=0,
\end{align*}
for all $x,y\in\cC$.

Recall that the  norm $\nup$  is isotropic if there is a non zero element $x\in \cC$ with $\nup(x)=0$, otherwise it is called anisotropic.
Note that any Cayley algebra with anisotropic norm is a division algebra.

It is known that, up to isomorphism, there is a unique Cayley algebra whose
norm is isotropic. It is called the split Cayley algebra. The split Cayley
$\cC$ admits a \emph{canonical basis} $\{e_1, e_2, u_1, u_2, u_3, v_1, v_2, v_3\}$.  The multiplication table in this basis is given in following Table
	\ref{table:1} (see \cite[\S 4.1]{EKmon} or  \cite[Section 2]{AEK}):
\begin{table}[h!]
	\caption{}
	\label{table:1}\vspace*{-6pt}
	\centering
	\begin{tabular}{ | c | c c | c c c | c c c |}
		\hline
		& $e_1$ & $e_2$ & $u_1$ & $u_2$ & $u_3$ & $v_1$ & $v_2$ & $v_3$ \\
		\hline
		$e_1$ & $e_1$ & 0  & $u_1$ & $u_2$ & $u_3$ & 0 & 0 & 0 \\
		$e_2$ & 0 & $e_2$  & 0 & 0 &  0 &  $v_1$ & $v_2$ & $v_3$ \\
		\hline
		$u_1$ & 0  & $u_1$ & 0 & $v_3$ & $-v_2$  & $-e_1$ & 0 & 0 \\
		$u_2$ & 0 & $u_2$  & $-v_3$ & 0 & $v_1$ & 0 & $-e_1$ & 0 \\
		$u_3$ & 0 & $u_3$  & $v_2$ & $-v_1$ & 0 & 0 & 0 & $-e_1$ \\
		\hline
		$v_1$ & $v_1$ & 0  & $-e_2$ & 0 & 0 & 0 & $u_3$  & $-u_2$ \\
		$v_2$ & $v_2$ & 0  & 0 & $-e_2$ & 0 & $-u_3$ & 0 & $u_1$ \\
		$v_3$ & $v_3$ & 0  & 0 & 0     & $-e_2$ & $u_2$ & $-u_1$ & 0  \\
		\hline
	\end{tabular}
\end{table}

Recall also that $\cK=\FF e_1+ \FF e_2$, which is isomorphic to $\FF\times \FF$, is the split Hurwitz
algebra of dimension $2$, and with $\cU = \FF u_1+\FF u_2+\FF u_3$ and $\cV = \FF v_1+\FF v_2+\FF v_3$,
the decomposition $\cC=\cK\oplus \cU\oplus \cV$ is a $\ZZ_3$-grading:
$\cC_{\overline{0}}=\cK, \, \cC_{\overline{1}}=\cU,\, \cC_{\overline{2}}=\cV$.

Let $\cC$ be a Cayley (or octonion) algebra over an arbitrary field $\FF$.
Note that any automorphism  $\varphi$ on $\mathcal{C}$ satisfies  $\varphi(1)=1,$  it leaves
invariant the subspace of traceless octonions $\cC_0=\{x \in \cC: \nup(x,1)=0\}$ and
\begin{align}\label{autsym}
\nup(\varphi(x))=\nup(x),\,\,\, \nup(\varphi(x), \varphi(y))=\nup(x, y)
\end{align}
for all $x, y\in \cC$.

\medskip

\section{Local automorphisms}

The following result is proved in \cite[Lemma 3.3]{AEK}. Recall that the trace of an
element $x$ of a Cayley algebra $\cC$ with norm $\nup$ is
$\nup(x,1)$.

\begin{lemma}\label{le:orbits}
	Let $\cC$ be a Cayley  over a field $\FF$ with the norm $\nup$. Any two elements of
	$\cC\setminus \FF 1$ are conjugate under $\Aut\cC$ if and only if they have the same norm and the same trace.
\end{lemma}

Let $\LocAut(\cC)$ be the set of local automorphisms of a Cayley algebra $\cC$ over a field $\FF$. Let $\nup$ be the norm on $\cC$. Denote by $\Ort(\cC,\nup)$ the corresponding orthogonal group.

\begin{theorem}\label{th:local_autos}
Let $\cC$ be a Cayley algebra  over a field $\FF$ with norm $\nup$. Then the set 
$\LocAut(\cC)$ coincides with
$\{\varphi\in\Ort(\cC,\nup)\mid \varphi(1)=1\}$. 
\end{theorem}

\begin{proof}
Any local automorphism fixes the unity $1$ of $\cC$ and preserves the norm of any element, because  so does any automorphism (see \eqref{autsym}).

Conversely, given an orthogonal transformation
$\varphi\in\Ort(\cC,\nup)$ with $\varphi(1)=1$, and given any element
$x\in\cC\setminus\FF 1$ we have $\nup\bigl(\varphi(x)\bigr)=\nup(x)$, and
$\nup\bigl(\varphi(x),1\bigr)=\nup\bigl(\varphi(x),\varphi(1)\bigr)
=\nup(x,1)$. Hence $\varphi(x)$ and $x$ have the same norm and trace, and the result follows from Lemma \ref{le:orbits}.
\end{proof}


If $\chr\FF\neq 2$, then $\{\varphi\in\Ort(\cC,\nup)\mid \varphi(1)=1\}$ is naturally isomorphic to the orthogonal group
$\Ort(\cC_0,\nup)$, where $\cC_0$ is the subspace of trace zero elements, that is, the orthogonal subspace to $\FF 1$.

However, if $\chr\FF=2$ there is the natural
group homomorphism
\[
\Phi:\{\varphi\in\Ort(\cC,\nup)\mid \varphi(1)=1\}\rightarrow\Ort(\cC_0,\nup)
\]
obtained by restriction. Take an element $a\in\cC$ with $\nup(a,1)=1$. Then
$\cK=\FF 1+\FF a$ is a composition subalgebra of $\cC$. Write $W=\cK^\perp$, so that
$\cC_0=\FF 1\oplus W$. The kernel of $\Phi$ consists of those elements
$\varphi\in\Ort(\cC,\nup)$ such that $\varphi(1)=1$ and $\varphi(w)=w$ for all
$w\in W$. Then $\varphi(a)$ is orthogonal to $\varphi(W)=W$, and
$\nup\bigl(\varphi(a),1\bigr)=\nup\bigl(\varphi(a),\varphi(1)\bigr)=\nup(a,1)=1$.
Hence $\varphi(a)=a+\mu 1$ for some $\mu\in\FF$, and since $a$ and $\varphi(a)$
have the same norm, we must have $\mu+\mu^2=1$, so either $\mu$ is $0$ or $1$.
In other words, the kernel of $\Phi$ is cyclic of order $2$. Moreover,
$\Phi$ is not surjective in general.

Actually,
as $\{x\in\cC_0\mid \nup(x,\cC_0)=0\}=\FF 1$, any element
$\phi$ in $\Ort(\cC_0,\nup)$ fixes $1$ and takes any element $w\in W$ to an element
of the form $\alpha(w)1+\sigma(w)$, for a linear map
$\alpha:W\rightarrow\FF$, and
an element $\sigma$ in the symplectic group $\SP(W,\nup)$ of $W$ relative to the alternating
bilinear form given by the polarization $\nup(.,.)$. Moreover, $\sigma$ and $\alpha$
are related by the condition
$\nup(w)-\nup\bigl(\sigma(w)\bigr)=\alpha(w)^2$ for all
$w\in W$, so $\alpha$ is determined by $\sigma$. The nondegeneracy of $\nup$ gives a unique element $w_\sigma\in W$ such that
$\alpha(w)=\nup(w_\sigma,w)1$ for all $w\in W$.

If $\FF$ is perfect, then given any $\sigma\in\SP(W,\nup)$, the map
$\gamma:w\to\nup(w)-\nup\bigl(\sigma(w)\bigr)$ is additive and `semilinear': $\gamma(\mu w)=\mu^2\gamma(w)$ for all $\mu\in\FF$.
Recall that by the definition of  a perfect field, $\mu\to\mu^2$ is an automorphism (the Frobenius automorphism) of $\FF$. It follows that, for a perfect $\FF$, the map
$\gamma$ is always of the form $\alpha(w)^2$ for a linear form
$\alpha$ and, therefore, the map $\phi\to\sigma$ gives a
well-known group isomorphism
$\Ort(\cC_0,\nup)\rightarrow \SP(W,\nup)$. (All this is valid for
odd-dimensional regular quadratic forms).

Moreover, for an arbitrary field $\FF$ of characteristic $2$, take 
$\phi\in\Ort(\cC_0,\nup)$, and $\sigma$, $w_\sigma$
as above: $\phi(w)=\sigma(w)+\nup(w_\sigma,w)1$
and $\nup(w)-\nup\bigl(\sigma(w)\bigr)=\nup(w_\sigma,w)^2$
for all $w\in W$. If $\phi$ is extended to an element of
$\Ort(\cC,\nup)$, then $\nup\bigl(\phi(a),1\bigr)=\nup(a,1)=1$, so 
$\phi(a)=\mu 1+a+w_a$
for a scalar $\mu\in \FF$ and an element $w_a\in W$.
From $\nup\bigl(\phi(a)\bigr)=\nup(a)$ we get $\nup(w_a)=\mu^2+\mu$, and from
\begin{align*}
0 & =\nup(a,w)=\nup\bigl(\phi(a),\phi(w)\bigr)
=\nup\bigl(\mu 1+a+w_a,\sigma(w)+\nup(w_\sigma,w)1\bigr)\\
 & =  \nup(w_\sigma,w)+\nup\bigl(w_a,\sigma(w)\bigr)
 =\nup\bigl(\sigma(w_\sigma),\sigma(w)\bigr)+\nup\bigl(w_a,\sigma(w)\bigr)\\
 &  =\nup\bigl(\sigma(w_\sigma)+w_a,\sigma(w)\bigr)
\end{align*}
for all $w\in W$, we get $w_a=\sigma(w_\sigma)$. Now we have 
\[
\mu^2+\mu=\nup(w_a)=\nup\bigl(\sigma(w_\sigma)\bigr)=
\nup(w_\sigma)+\nup(w_\sigma,w_\sigma)^2=\nup(w_\sigma).
\]
Therefore, the obstruction for $\phi\in\Ort(\cC_0,\nup)$ to extend 
to an element in $\Ort(\cC,\nup)$ is that $\nup(w_\sigma)$ must belong to
$\{\mu+\mu^2\mid \mu\in\FF\}$. In particular, $\Phi$ is surjective
if $\FF$ is algebraically closed.

\smallskip

The situation for Lie algebras is easier: $\{d\in\frso(\cC,\nup)\mid d(1)=0\}$ is always isomorphic to $\frso(\cC_0,\nup)$ (see \cite[Remark 3.2]{AEK}).

\medskip

\section{$2$-local automorphisms}

In order to deal with $2$-local automorphisms, we need a slight variation of the proof of \cite[Lemma 3.2]{AKA}.

\begin{lemma}\label{le:2local_linear}
Let $\cC$ be a Cayley  over a field $\FF$ with norm $\nup$. Then any
$2$-local automorphism of $\cC$ is linear.
\end{lemma}
\begin{proof}
Consider a $2$-local automorphism $\varphi\in\twoLocAut(\cC)$.
For any two elements $x,y\in\cC$, there is an automorphism
$\psi\in\Aut(\cC)$ such that $\varphi(x)=\psi(x)$, $\varphi(y)=\psi(y)$.
As any automorphism preserves the norm, we get
\begin{equation}\label{eq:2local_norm}
\nup\bigl(\varphi(x)\bigr)=\nup(x)\quad\text{and}\quad
\nup\bigl(\varphi(x),\varphi(y)\bigr)=\nup(x,y)
\end{equation}
for any $x,y\in\cC$.

Pick an arbitrary
basis $\{x_i\mid 1\leq i\leq 8\}$ of $\cC$. The matrix
$\Bigl(\nup\bigl(\varphi(x_i),\varphi(x_j)\bigr)\Bigr)
=\Bigl(\nup(x_i,x_j)\Bigr)$ is nondegenerate, and hence
$\{\varphi(x_i)\mid 1\leq i\leq 8\}$ is another basis of $\cC$. In particular, $\varphi(\cC)$ spans the whole $\cC$.

For $x,y,z\in\cC$, using \eqref{eq:2local_norm} we obtain
\begin{align*}
\nup\bigl(\varphi(x+y),\varphi(z)\bigr)
& =\nup(x+y,z)=\nup(x,z)+\nup(y,z)\\
& =\nup\bigl(\varphi(x),\varphi(z)\bigr)
     +\nup\bigl(\varphi(y),\varphi(z)\bigr)
=\nup\bigl(\varphi(x)+\varphi(y),\varphi(z)\bigr).
\end{align*}
Hence $\varphi(x+y)-\varphi(x)-\varphi(y)$ is orthogonal to all the
elements in $\varphi(\cC)$, and this spans the whole $\cC$.
The nondegeneracy of $\nup$ forces
$\varphi(x+y)=\varphi(x)+\varphi(y)$.

On the other hand, for any $x\in \cC$ and $\lambda\in\FF$, there is
an automorphism $\psi\in\Aut(\cC)$ such that $\varphi(x)=\psi(x)$
and $\varphi(\lambda x)=\psi(\lambda x)$, so that
$\varphi(\lambda x)=\psi(\lambda x)
=\lambda\psi(x)=\lambda\varphi(x)$, and we conclude that $\varphi$ is linear.
\end{proof}

In particular, any $2$-local automorphism of a Cayley algebra is a
local automorphism, and hence it belongs to
$\{\varphi\in\Ort(\cC,\nup)\mid \varphi(1)=1\}$.

The proof of the next lemma is straightforward. The similar result
for local automorphisms is valid too.

\begin{lemma}\label{le:auto2local}
Let $\varphi$ be a $2$-local automorphism and $\psi$ an automorphism
 of a nonassociative algebra $\cA$. Then $\psi\varphi$ and
$\varphi\psi$ are
$2$-local automorphisms too.
\end{lemma}

The situation in the split case is quite simple:

\begin{theorem}\label{th:2local_split}
Let $\cC$ be the split Cayley algebra over a field $\FF$ with norm
$\nup$. Then any $2$-local automorphism of $\cC$ is an automorphism:
$\twoLocAut(\cC)=\Aut(\cC)$.
\end{theorem}
\begin{proof}
Consider a canonical basis $\{e_1,e_2,u_1,u_2,u_3,v_1,v_2,v_3\}$
as in Table~\ref{table:1}. Let $\varphi$ be a $2$-local automorphism
of $\cC$, and pick an automorphism $\psi$ such that
$\varphi(e_1)=\psi(e_1)$ and $\varphi(e_2)=\psi(e_2)$. Then the
$2$-local automorphism $\varphi'=\psi^{-1}\varphi$ fixes both $e_1$
and $e_2$.

The subspace $\cU=\FF u_1+\FF u_2+\FF u_3$ is the `Peirce component'
$\{u\in \cC\mid e_1u=u=ue_2\}$. For $u\in \cU$, let $\phi$ be an
automorphism such that $e_1=\varphi'(e_1)=\phi(e_1)$ and
$\varphi'(u)=\phi(u)$. Then $\phi(e_2)=\phi(1-e_1)=1-e_1=e_2$, and
hence $\phi$ preserves the Peirce component $\cU$. In particular
$\varphi'(\cU)=\cU$. In the same vein, the Peirce component
$\cV=\FF v_1+\FF v_2+\FF v_3=\{v\in \cV\mid e_2v=v=ve_1\}$ is
fixed too by $\varphi'$.

Since $\cU$ and $\cV$ are isotropic subspaces, $\{u_1,u_2,u_3\}$ and $\{v_1,v_2,v_3\}$ are dual bases relative to
$\nup$, if $A=\bigl(\alpha_{ij}\bigr)_{1\leq i,j\leq 3}$ is the coordinate matrix of
$\varphi'\vert_\cU$ in the basis $\{u_1,u_2,u_3\}$, then the coordinate
matrix of $\varphi'\vert_\cV$ in the basis $\{v_1,v_2,v_3\}$ is
$(A^t)^{-1}$. ($A^t$ denotes the transpose of $A$).

If $\det(A)=1$, then $\varphi'$ is an automorphism and we are done.

Otherwise,  if $\det(A)= \lambda\neq 1$, we can factor $A$ as
\[
A=\bar A \begin{pmatrix} 1&0&0\\ 0&1&0\\
 0&0&\lambda\end{pmatrix}
\]
for a matrix $\bar A$ of determinant $1$. If the  linear map $\psi$  fixes $e_1$, $e_2$ and the subspaces $\cU$ and $\cV$,
and such that the coordinate matrix of $\psi\vert_\cU$ in the above basis
is $\bar A^{-1}$ and the coordinate matrix of $\psi\vert_\cV$ is
$\bar A^t,$ then $\psi$  is an automorphism. Therefore  the $2$-local automorphism
$\varphi''=\psi\varphi'$ fixes the elements $e_1,e_2,u_1,u_2,v_1,v_2$
and sends $u_3\to \lambda u_3$, $v_3\to\lambda^{-1}v_3$.
Pick now the elements $x=v_1-u_1$ and $y=u_3+v_2$. Then $xy=y$
and $x$ is fixed by $\varphi''$. There is an automorphism
$\tau\in\Aut(\cC)$ such that $\tau(x)=\varphi''(x)=x$ and
$\tau(y)=\varphi''(y)=\lambda u_3+v_2$. But a simple computation gives
\[
\lambda u_3+v_2=\tau(y)=\tau(xy)=\tau(x)\tau(y)=x\tau(y)
=(v_1-u_1)(\lambda u_3+v_2)=u_3+\lambda v_2,
\]
and this is a contradiction as we have $\lambda \neq 1$.
\end{proof}

\medskip

We turn our attention now to the case of division Cayley algebras.
Here, as was the case for $2$-local derivations, the situation is different.

\begin{lemma}\label{le:2local_division_K}
Let $\cC$ be a division Cayley algebra over a field $\FF$ with norm
$\nup$. Let $\cK$ be a two-dimensional composition subalgebra of
$\cC$ (that is, $\nup\vert_\cK$ is regular), let $x$ be an arbitrary
element of $\cC$, and let $\varphi$ be a
local automorphism of $\cC$. Then there is an automorphism $\psi$
of $\cC$ such that $\psi\vert_\cK=\varphi\vert_\cK$ and
$\psi(x)=\varphi(x)$.
\end{lemma}
\begin{proof}
Recall that Theorem \ref{th:local_autos} gives
$\LocAut(\cC)=\{\varphi\in\Ort(\cC,\nup)\mid \varphi(1)=1\}$.

Pick an element $a\in\cK\setminus\FF 1$. There is an automorphism
$\phi\in\Aut(\cC)$ such that $\phi(a)=\varphi(a)$, and hence the
local automorphism $\varphi'=\phi^{-1}\varphi$ fixes $\cK$ elementwise.

Write $x=b+y$ with $b\in\cK$ and $y\in\cK^\perp$. If $y=0$, we are done as $\varphi$ and $\phi$ coincide on $\cK$ and hence also on $x$. Otherwise,
as $\cC$ is a division algebra $\nup(y)\neq 0$. Since $\varphi'$ is
an orthogonal transformation, $\varphi'(y)$ lies too in $\cK^\perp$, and
it has the same norm as $y$. As in the proof of
\cite[Theorem 1.7.1]{SpringerVeldkamp}, the Cayley-Dickson doubling process, together with Witt's Cancellation Theorem,
shows that there is an automorphism $\tau\in\Aut(\cC)$ which is the
identity on $\cK$ and such that $\tau(y)=\varphi'(y)$. Then
$\tau(x)=\tau(b+y)=b+\tau(y)=b+\varphi'(y)=\varphi'(x)
=\phi^{-1}\varphi(x)$. The automorphism $\psi=\phi\tau$ coincides
with $\varphi$ both on $\cK$ and on $x$.
\end{proof}

This result settles the situation easily if the characteristic
of the ground field is not $2$. Over fields of characteristic $2$, the problem is more difficult. Note that the only Cayley (or quaternion) algebra
over a perfect field of characteristic $2$ is the split one, so we are
forced to deal with
fields $\FF$ of characteristic $2$ with $\FF\neq \FF^2$.

\begin{theorem}\label{th:2localautos_division_not2}
Let $\cC$ be a division Cayley algebra over a field $\FF$. Then any local
automorphism is a $2$-local automorphism: $\LocAut(\cC)=\twoLocAut(\cC)$.
\end{theorem}

\begin{proof}
Let $\varphi$ be a local automorphism of $\cC$, and let $x,y\in\cC$.
If $x\in\FF 1$, then any automorphism $\psi$ such that
$\varphi(y)=\psi(y)$ satisfies $\varphi(x)=\psi(x)$ too. Therefore we
assume from now on that $x,y$ are not in $\FF 1$.

If the characteristic of $\FF$ is not $2$, $\cK=\FF 1+\FF x$ is a composition subalgebra of $\cC$, because
$\nup$ is anisotropic. Lemma \ref{le:2local_division_K} shows that there is an automorphism $\psi\in\Aut(\cC)$ such that $\varphi$ and
$\psi$ coincide on both $x$ and $y$. The same argument works
if the characteristic of $\FF$ is $2$ and $\nup(x,1)\neq 0$ (or $\nup(y,1)\neq 0$),
as this forces $\cK=\FF 1+\FF x$ to be a composition subalgebra.

Hence assume that $\chr\FF=2$ and $\nup(x,1)=0=\nup(y,1)$. In particular
$\bar x=-x=x$ and $\bar y=-y=y$. We may also assume that $1$, $x$ and $y$ are
linearly independent. We are left with two
different cases, depending on $\nup(x,y)$ being $0$ or not.
Write $\alpha=\nup(x)$,
$\beta=\nup(y)$, and $\gamma=\nup(x,y)$, and note that $\alpha\neq 0\neq\beta$, since $\nup$ is anisotropic, and $\nup(xy)=\alpha\beta$.

If $\nup(x,y)\neq 0$, then $\nup(xy,1)=\nup(x,\bar y)=\nup(x,y)\neq 0$, so that
$\cK=\FF 1+\FF xy$ is a composition subalgebra of $\cC$.  The multiplicative
property of $\nup$ gives $\nup(x,xy)=\nup(x)\nup(1,y)=0$, so $x$ is orthogonal
to $\cK$, and hence $\cQ=\cK\oplus\cK x$ is a quaternion subalgebra of $\cC$.
Besides, using that $xy+yx=x\bar y+y\bar x=\nup(x,y)1$, we obtain
\[
\begin{split}
x(xy)&=x^2y=\nup(x)y=\alpha y\\
(xy)x&=\bigl(\nup(x,y)1-yx)x=\gamma x-yx^2=\gamma x+\alpha y,
\end{split}
\]
so that $\{1,xy,x,y\}$ is a basis of $\cQ$, and the multiplication on $\cQ$ is completely determined in the following Table \ref{table:Q}.
\begin{table}[h!]
	\caption{}\label{table:Q}
	\vspace*{-6pt}\begin{tabular}{ | c | c c  c c |}
		\hline
		& $1$ & $xy$ & $x$ & $y$  \\
		\hline
		$1$ & $1$ & $xy$ & $x$ & $y$ \\
  $xy$ & $xy$ & $\gamma xy+\alpha\beta1$ & $\gamma x+\alpha y$ & $\beta x$ \\
  $x$ & $x$  & $\alpha y$ & $\alpha 1$ & $xy$  \\
  $y$ & $y$ & $\beta x+\gamma y$  & $\gamma 1+xy$ & $\beta 1$ \\
		\hline
	\end{tabular}
\end{table}

Since $\varphi$ is an orthogonal transformation that fixes $1$,
the elements $x'=\varphi(x)$ and $y'=\varphi(y)$ also satisfy the conditions
$\nup(x',1)=0=\nup(y',1)$, $\nup(x')=\alpha$, $\nup(y')=\beta$,
$\nup(x',y')=\gamma$, and hence the quaternion subalgebra
$\cQ'=\FF 1+\FF x'y'+\FF x'+\FF y'$ has the same multiplication table as in
Table \ref{table:Q}. Therefore there is an isomorphism $\psi:\cQ\rightarrow\cQ'$
with $\psi(x)=x'=\varphi(x)$ and $\psi(y)=y'=\varphi(y)$.
Besides, using the Cayley-Dickson doubling
process (as in \cite[Corollary 1.7.3]{SpringerVeldkamp}), $\psi$ extends to
an automorphism of $\cC$, and we are done in this case.

Finally, assume $\nup(x,y)=0$. In this case we have $xy=yx$ and
$\nup(xy,1)=\nup(x,y)=0$,
$\nup(xy,x)=\nup(x)\nup(y,1)=0$, and $\nup(xy,y)=\nup(x,1)\nup(y)=0$, so the
subspace $\cS=\FF 1+\FF x+\FF y+\FF xy$ satisfies $\nup(\cS,\cS)=0$, that is, $\cS$ is totally isotropic for the polar bilinear form. The dimension of $\cS$ is $4$, because otherwise we would have $xy=\delta 1+\mu x+\nu y$ for some scalars
$\delta,\mu,\nu\in\FF$, and this would give
$y=(x-\nu 1)^{-1}(\delta 1+\mu x)\in\FF 1+\FF x$, a contradiction with $1,x,y$
being linearly independent.
Since $\varphi$ is an orthogonal
transformation that fixes $1$, the same conditions hold for $x'=\varphi(x)$ and
$y'=\varphi(y).$ So that $\cS'=\FF 1+\FF x'+\FF y'+\FF x'y'$ is again totally isotropic for the polar form. Since $\nup(x)=\nup(x')$, $\nup(y)=\nup(y')$,
and $\nup(xy)=\nup(x)\nup(y)=\nup(x')\nup(y')=\nup(x'y')$, the subspaces $\cS$ and
$\cS'$ are isometric, relative to the quadratic form $\nup$, with an isometry sending
$1\to 1$, $x\to x'$, $y\to y'$, and $xy\to x'y'$. Witt's Extension Theorem
(see, for instance, \cite[Theorem 8.3]{EKM}) shows that there is an orthogonal
transformation $\sigma\in\Ort(\cC,\nup)$ such that $\sigma(1)=1$,
$\sigma(x)=x'=\varphi(x)$, $\sigma(y)=y'=\varphi(y)$, and
$\sigma(xy)=x'y'=\varphi(x)\varphi(y)$. Pick an element $u$ such that
$\nup(u,x)=\nup(u,y)=\nup(u,xy)=0$ and $\nup(u,1)=1$. This is always possible
because $1,x,y,xy$ are linearly independent and $\nup$ is nondegenerate.
Write $u'=\sigma(u)$ and consider the composition subalgebras $\cK=\FF 1+\FF u$ and
$\cK'=\FF 1+\FF u'$. The restriction of $\sigma$ to $\cK$ gives an algebra
isomorphism $\psi:\cK\rightarrow\cK'$. Note that $x$ is orthogonal to $\cK$, so that
$\cQ=\cK\oplus\cK x$ is a quaternion subalgebra and, in the same vein,
$\cQ'=\cK'\oplus\cK'x'$ is another quaternion subalgebra. The Cayley-Dickson
doubling process shows that $\psi$ can be extended to an isomorphism
$\cQ\rightarrow\cQ'$, also denoted by $\psi$, that satisfies
$\psi(x)=x'=\varphi(x)$. Finally, $y$ is orthogonal to $\cK$
and $\nup(\cK x,y)=\nup(\cK,y\bar x)=\nup(\cK,xy)=0$, so that $y$ is orthogonal to
$\cQ$, and hence we get $\cC=\cQ\oplus\cQ y$.
In  a similar way $\cC=\cQ'\oplus\cQ'y'$, and the isomorphism
$\psi:\cQ\rightarrow \cQ'$ can be extended to an automorphism of $\cC$ with
$\psi(y)=y'=\varphi(y)$.
\end{proof}

\medskip

\section{Open problems}

In this section we will restrict ourselves to the real or complex field, but similar arguments can be discussed over arbitrary fields, if one replaces Lie group by
affine algebraic group.

Let $\cA$ be a finite dimensional (not necessary associative) algebra over the field $\FF=\RR$ or $\CC$ and denote by $\Aut(\cA)$ the group of its automorphisms. Let $G$ be a Lie group and let $\textrm{Lie}(G)$ be the its tangent Lie algebra (see for details \cite{EKmon, H15}).

It is well-known that
$\textrm{Lie}(\Aut(\cA)) \cong \Der(\cA)$ (see \cite[Page 316]{EKmon}).

Note that in the case $\FF=\RR$ or $\CC,$ by Theorem~\ref{th:local_autos}, it follows that $\LocAut(\cC)= \{\varphi\in\Ort(\cC,\nup)\mid \varphi(1)=1\}$ is a Lie group, because $\{\varphi\in\Ort(\cC,\nup)\mid \varphi(1)=1\}$ is a closed subgroup in the Lie group $\Ort(\cC,\nup).$
In \cite[Theorem 3.1]{AEK} we have proved that the space of all local derivations of a Cayley algebra $\cC$ is the Lie algebra
$\{d\in\frso(\cC,\nup)\mid d(1)=0\}.$  Thus we obtain that
$\textrm{Lie}(\LocAut(\cC)) \cong \LocDer(\cC).$ These observations lead us formulate the following  problems:
\begin{problems}\label{LA}

\

\begin{enumerate}
\item Is $\LocAut\left(\cA\right)$ a Lie group?
\item Is	$\left(\LocDer\left(\cA\right), [\cdot, \cdot]\right)$ a Lie algebra?
\item If the above two assertions are true, are the Lie algebras 
$\textrm{Lie}\left(\LocAut(\cA)\right)$ and $\LocDer\left(\cA\right)$ isomorphic?
\end{enumerate}
\end{problems}

It is well known that the space of all derivations $\Der(\cA)$ is a Lie algebra with respect to the Lie bracket. At the same time, it is not clear whether the space of all local derivations $\LocDer(\cA)$ forms a Lie algebra.

As we have already noted  for Cayley algebras, Problems~\ref{LA} have a positive solution. Below we list some other classes of algebras for which Problems~\ref{LA} also have a positive solution:
\begin{itemize}
\item finite-dimensional complex simple Lie algebras \cite[Theorem 3.1]{AK16}, \cite[Theorem 3.14]{C19}. In fact, since
$
\LocAut(\mathfrak{g})=\Aut(\mathfrak{g})\ltimes \{\textrm{id}_\mathfrak{g}, -\textrm{id}_\mathfrak{g}\}
$ for a simple Lie algebra $\mathfrak{g},$ we have that
$$
\textrm{Lie}\left(\LocAut(\mathfrak{g})\right)=\textrm{Lie}\left(\Aut(\mathfrak{g})\right)\cong\Der\left(\mathfrak{g}\right)= \LocDer\left(\mathfrak{g}\right).
$$
\item complex  Leibniz  algebras of the form $\mathfrak{g}=\mathfrak{sl_n}\ltimes \cI,$ where $\mathfrak{sl_n}$ is the special linear Lie algebra  and $\cJ$ is the subspace generated by squares of
$\mathfrak{g}$  (see \cite[Theorem 20]{AKO20}, \cite[Theorem 4]{AK18}).
\item the algebra $NT(3,\FF)$ of lower niltriangular matrices of order $3$ over $\FF$ (see \cite[Theorem 2]{E11}).
\end{itemize}
In the  last two cases, as in the first one, the isomorphism $\textrm{Lie}\left(\LocAut(\cA)\right)\cong\LocDer\left(\cA\right)$ follows directly from the descriptions of $\LocAut(\cA)$ and $\LocDer\left(\cA\right).$

\end{document}